\documentclass[12pt,reqno]{amsart}

\date{\today}
\usepackage[hyphens]{url}
\usepackage{bm} 
\usepackage{nicefrac}
\usepackage{yhmath}
\usepackage{accents}
\usepackage{hyperref}
\usepackage{xcolor,tikz}
\usetikzlibrary{arrows,chains,matrix,positioning,scopes}
\usepackage[numbers, comma, sort&compress]{natbib}
\hypersetup{ colorlinks=true, linkcolor=blue, citecolor=blue,
 filecolor=blue, urlcolor=blue}

\usepackage{fancyhdr,latexsym,amsmath,amsfonts,amssymb,amsbsy,amsthm,url,mathrsfs,bbm}
\usepackage{amscd}

\allowdisplaybreaks
\newtheorem{theorem}{Theorem}[section]
\newtheorem{lemma}{Lemma}[section]
\newtheorem{claim}{Claim}[section]

\newtheorem{corollary}{Corollary}[section]

\theoremstyle{definition}

\theoremstyle{remark}
\newtheorem{remark}{Remark}[section]

\numberwithin{equation}{section}
\numberwithin{equation}{section}

\DeclareMathOperator{\R}{\mathbb{R}}

\DeclareMathOperator{\N}{\mathbb{N}}

\DeclareMathOperator{\one}{\mathbbm{1}} 

\newcommand{\De}{\mathrm d}
\newcommand{\E}[1]{\mathsf{E}\left[#1 \right]}

\newcommand{\var}[1]{\mathsf{Var}\left(#1 \right)}

\newcommand{\prob}[1]{\mathsf{P}\left(#1 \right)}
\newcommand{\abs}[1]{\left| #1 \right|} 
\newcommand{\norm}[1]{\left\lvert\! \left\lvert#1 \right\rvert\!\right\rvert}

\DeclareMathOperator{\e}{e}
\renewcommand{\O}[1]{\mathrm{O}\left(#1\right)} 
\renewcommand{\o}[1]{\mathrm{o}\left(#1\right)} 
\newcommand{\la}{\left<}
\newcommand{\ra}{\right>}

\DeclareMathOperator{\diam}{\mathrm{diam}}
\DeclareMathOperator{\dime}{\mathrm{dim}}

\newcommand{\f}{\frac}  

\newcommand{\eps}{\epsilon}

\newcommand{\blank}[1]{}
\renewcommand{\i}{{i\mkern1mu}}

\newcommand{\eq}[1]{\begin{equation#1}}
\newcommand{\eeq}[1]{\end{equation#1}}
\newcommand{\eqa}[1]{\begin{eqnarray#1}}
\newcommand{\eeqa}[1]{\end{eqnarray#1}}

\newcommand{\Ss}{\mathcal S(\mathbb R^d)}
\newcommand{\Sd}{\mathcal S^\prime(\mathbb R^d)}

\definecolor{magenta}{rgb}{1.0, 0.0, 0.56}

\begin{document}

\title[Thick points]{Thick points for Gaussian free fields with different cut-offs}
\author[A. Cipriani]{Alessandra Cipriani}
\address{Weierstra{\ss}-Institut, Mohrenstra{\ss}e 39, 10117, Berlin, Germany}
\email{Alessandra.Cipriani@wias-berlin.de}

\author[R. S. Hazra]{Rajat Subhra Hazra}\thanks{This project was completed during the stay of Rajat Subhra Hazra at the University of Zurich, and hence he acknowledges the Institute of Mathematics for its kind support.}
\address{Institut f\"ur Mathematik\\ Universit\"at Z\"urich\\ Winterthurerstrasse 190\\ 8057-Zurich, Switzerland}
\email{rajat.hazra@math.uzh.ch}



\begin{abstract}
Massive and massless Gaussian free fields can be described as generalized Gaussian processes indexed by an appropriate space of functions. In this article we study various approaches to approximate these fields and look at the fractal properties of the thick points of their cut-offs. Under some sufficient conditions for a centered Gaussian process with logarithmic variance we study the set of thick points and derive their Hausdorff dimension. We prove that various cut-offs for Gaussian free fields satisfy these assumptions. We also give sufficient conditions for comparing thick points of different cut-offs.
\end{abstract}

\maketitle
\section{Introduction}
Let $D\subseteq \R^d$ with $d\ge 2$ (possibly $D=\R^d$). A generalized Gaussian field 
(GGF) $X$ is a collection of centered Gaussian random variables indexed by 
a certain class of functions $H$, that is, the field can be written as $\{(X,f): \, f\in H\}$. 
$H$ is in the present paper a Hilbert space of functions on $D$, and more specifically a Sobolev space. 
Notable examples of such GGFs are the massive and 
massless Gaussian free fields (GFF), for which correlations blow up logarithmically in the distance between two points. 
The study of GFFs has received considerable attention in the context of 
statistical mechanics and physics, 
as they can be seen as multidimensional generalizations of Brownian motion (see 
\citet{Sheff} for their construction and properties).

The two most important places (among many) where such fields have shown 
prominence is the construction of the Liouville Quantum Gravity measure 
(\citet{DS10}) and the theory of Gaussian multiplicative chaos 
(\citet{Kah85, Revisited}). In both these cases one constructs a random measure 
on $D$ given by
\begin{equation}\label{eq:lqg}
m_\gamma(\De x)= \exp\left(\gamma X(x)-\frac{\gamma^2}{2}\E{X(x)^2}\right)\De x,\quad \gamma \ge 0.
\end{equation}
For log-correlated models \eqref{eq:lqg} is known as Gaussian Multiplicative chaos (GMC) measure 
(after \citet{Kah85}). In particular, when $X$ is a planar massless or massive free field it is known 
as Liouville Quantum Gravity measure. 
Since $X$ is not defined pointwise, this measure is merely formal. To avoid this discrepancy, in \eqref{eq:lqg} $X$ is replaced by a 
space-time centered Gaussian process $\{X_\epsilon(x):\, x\in D, \,\epsilon>0\}$, which converges to $X$ in law, and for which \eqref{eq:lqg} makes sense. Note that $X_\eps$ retains log-correlations in the sense that $\var{X_\eps(x)}$ behaves like 
$\log(1/\eps)$ as $\eps$ goes to $0$, but $X_\eps$ is a ``proper'' Gaussian process in $x$.
This approach is used extensively in literature 
(see \citet{DRSV2} for example) and is connected to the seminal work 
of \citet{Kah85}. 
We describe more explicitly some of these approximations in Subsection~\ref{subsec:cut-description}.

It is natural to ask whether different approximations almost surely give the same GMC. The question 
was already studied in \citet{Kah85, RhoVarRev} where, 
for certain cut-offs, the equality in law was proved. In the case of planar 
(massless) GFF, it was shown by \citet{DS10} that measures arising from the circle 
average process and by the orthonormal basis expansion of $H^1_0(D)$ are almost surely the same. For some recent development in the area we refer to~\citet{Shamov}.

In this article we continue the study of almost-sure universality of 
cut-offs with respect to thick points (the term was used in \citet{HMP}, and also referred to 
as multi-fractal behavior in \citet{Kah85}). This is the set of points which encapsulates the extremal behavior of the field. For a cut-off $X_\epsilon(x)$ the thick points are defined as 
\begin{equation}\label{def:a-thick}
T(a)=\left\{x\in D\colon \lim_{\eps\to 0}\frac{X_\eps(x)}{\var{X_\eps(x)}}=a\right\},\quad a\geq 0.
\end{equation}
Their importance  comes from the fact that they have full mass for the 
Gaussian multiplicative chaos measure, and give also information on the 
behavior of the so-called Liouville Brownian motion for the Liouville quantum gravity measure (\citet{GarRV14}).  Under a H\"older-type condition we show in Theorem~\ref{thm:upper}  that the Hausdorff dimension $\dime_H$ of $T(a)$ has an upper bound of $d-{a^2}/{2}$ when $a< \sqrt{2d}$. 
The H\"older-type condition seems to 
be a minimal requirement as these fields are not smooth and hence exhibit a fractal behavior. 
 
The Hausdorff dimension has a lower bound of $d-{a^2}/{2}$ if $a<\sqrt{2d}$ as computed by \citet{Kah85}. Hence in Theorem~\ref{thm:lower} we recall briefly the main steps of his proof and give precise assumptions that complement our upper bound and wrap up the question of the Hausdorff dimension of thick points for GGF with log-correlations.

In view of the above results, one might ask whether 
there is a possibility of comparing the extremal behavior for different 
processes. We give a partial answer to this query 
by imposing a sufficient condition 
(see Theorem~\ref{thm:comparison}) on the 
difference of two cut-offs which yields almost the sure equality of thick points. 

The outline of the article is as follows. 
In Section~\ref{sec: construction} we review the definitions 
of massive and massless GFF and in Subsection~\ref{subsec:cut-description} some
cut-offs procedures. Then we state the main results with brief descriptions in Subsection~\ref{sec:main}. In Section~\ref{examples} we first show that the 
examples considered 
in Subsection~\ref{subsec:cut-description} satisfy the assumptions of these results. Finally, 
in Section~\ref{sec:proof} we provide the proofs of our results.

\section{Construction of  free fields and approximations}\label{sec: construction}
\subsection{Two examples of fields}
\subsubsection{Massive free fields on $\R^d$}
Let $\R^d, d\ge 2$. Let $\Ss$ be the Schwartz space consisting of smooth functions whose derivatives decay faster than any polynomial. Let $\Sd$ be the space of tempered distribution which are also the continuous linear functionals on $\Ss$. Also, $\Ss$ form a dense subset of $\Sd$ with respect to the weak*-topology. With $C_0^\infty(D)$ we denote the set of smooth and compactly supported functions on $D$. To avoid somehow lengthy notation, we set $L^2=L^2(\R^d,\,\mathrm d x)$ dropping the reference measure. For $\xi\in \R^d$, let $\la \xi \ra_m=\left(m^2+\|\xi\|^2\right)^{1/2} $ and we denote by $\la \xi\ra= \la \xi \ra_1$. We will denote in some instances the scalar product in a Hilbert space $H$ by angle brackets with a subindex as $\la\,\cdot,\,\cdot\ra_H$. For $s\in \R$ 
we denote the operator  $B_m^s \colon \Ss\mapsto \Ss$ defined by
\begin{equation}\label{eq:besseldef}
B_m^s\phi(x)= \int_{\R^d} e^{-i\la \xi,x\ra_{\R^d}}\la \xi\ra_m^s \widehat\phi(\xi) \De \xi. 
\end{equation}
This corresponds to the definition of the (formal) Bessel operator $B_m^s\phi:=
\left(m^2\mathbb I-\Delta\right)^{-s/2}\phi.$
Let us denote $G_d(x)= K_0(m\|x\|)$, where $K_0(\cdot)$ is the modified Bessel function; it is well known (see \citet{Stein}) that $\widehat G_d(\xi)= \la \xi \ra_m^{-d}$ and hence one can write
$$B_m^{-d} \phi(x)= \int_{\R^d} G_d(x-y) \phi(y) \De y.$$
We want to look at generalized massive free fields indexed by $ f\in \Ss$ such that $$\E{(X,f)(X,g)}=\la f, B^{-d}g \ra_{L^2}.$$ It can be shown that the  functional $$L(\phi)= \exp\left(-\frac12 \la \phi, B_m^{-d}\phi \ra_{L^2} \right),$$
is a positive definite, continuous functional and hence it induces a measure on $\Sd$ whose characteristic functional is given by $L(\phi)$. This gives a generalized Gaussian field $\{( X,\,\phi),\,  \phi\in \Ss\}$ whose covariance can be represented by
\begin{equation}\label{eq:cov}
\E{(X,\phi)(X,\psi)}= \int \widehat\phi(\xi) \overline{\widehat{\psi(\xi)}}\la\xi \ra_m^{-d} \De \xi
\end{equation}
(see \citet{Hi93}, \citet{Yag57}).
The tempered measure $\mu(\De x)= \la \xi \ra_m^{-d} \De \xi$ can be realized as the spectral measure of the covariance of this Gaussian process. We remark here that the Hilbert space associated to the GGF $X$ is in this case the fractional Sobolev space $H^{d/2}(\R^d)$, that we recall being defined by 
$$H^s(\R^d):=\{ \phi \in \Ss: B^s\phi \in L^2(\R^d)\},\quad s\in \R.$$
For details on the construction of such generalized Gaussian fields, we refer the interested readers to \citet{GVbook}. See also for a white noise representation for the massive free fields in Subsection~\ref{subsec:wncutoff}.

\subsubsection{Massless planar Gaussian Free Field} Let $C_0^\infty(D)$ be space of smooth functions vanishing outside $D$, a bounded domain of $\R^2$. Let $H^1_0(D)$ be the Hilbert space which is the closure of 
$C_0^\infty(D)$ under the norm
$$\|f\|_{H^1}^2 =\int_D \|\nabla f (x)\|^2 \De x.$$
The dual of $H_0^1(D)$ is given by $H^{-1}(D)$ equipped with the norm
$$\|f\|_{H^{-1}}=\sup_{g\in C_0^\infty(D), \|g\|_{H^1}\le 1} \la f, g\ra,$$
where $\la\cdot ,\,\cdot\ra$ denotes the duality pairing. Note that for $f, g\in C_0^\infty(D)$, we have by Green's identity that 
$\la f, g\ra_{H^1}=\la f, \Delta g\ra_{L^2}$ and it follows that $\la f, g\ra_{H^{-1}}= \la f, \Delta^{-1}g\ra_{L^2}$,
where for $g\in C_0^\infty(D)$ one denotes
$$\Delta^{-1}g(x)= \int_{D}G_D(x,y) g(y) \De y.$$
Here $G_D(x,y)$ is the Green's function for the Dirichlet problem on a planar
domain and it is well known that
\begin{equation}\label{eq:green1}
G_D(x,y)= 2\pi \int_0^\infty p_D(t,\,x,\,y)\De t.
\end{equation}
$p_D$ is the transition kernel of standard Brownian motion killed at exiting $D$.

A (massless) Gaussian free field can described as a centered Gaussian process indexed by $H^{-1}(D)$, that is, a collection
$\{(\Phi, f): \, \, f\in H^{-1}(D)\}$ such that $$\mathrm{Cov}\left( (\Phi, f)(\Phi,g)\right)= \la f, g\ra_{H^{-1}}.$$
If we restrict ourselves to $C_0^\infty(D)$ it can be shown that
$$\mathrm{Cov}\left((\Phi,f)(\Phi,g)\right)= \int_D\int_D f(x) g(y) G_D(x,y) \De x \De y.$$
If $D=[0,1]^2$ the Gaussian Free Field has a formal representation as 
\eq{}\label{eq:onb}\Phi= \sum_{j,\,k\in\N}X_{j,k}e_{j,k}\eeq{}
where
$e_{j,k}$ are eigenfunctions given explicitly by
$e_{j,k}(x,y)= 2 \left(\sin(\pi jx) \sin(\pi k y)\right)\left(j^2+k^2\right)^{-1/2},$
which also form an orthonormal basis of $H^1_0(D)$.  Thus $\Phi$ converges almost surely in $H^{-1}(D)$ as remarked above.
We refer the readers for a more detailed construction to \citet{Dubedat} and \citet{Sheff}.
\subsection{The construction of cut-offs}\label{subsec:cut-description}
There are several ways in which one can approach the question of approximating a field with infinite variance by cut-offs. We will list here only a few of those examples. 
\subsubsection{White-noise cut-offs for massive free fields}\label{subsec:wncutoff}
Let $W$ be a Gaussian complex white noise with control measure $\mu(\De \xi)=\la \xi \ra_m^{-d}\De \xi$.  Formally, the field $X$ is given by the characteristic function of the white noise. That is, if $\zeta(\lambda,\,\xi)= e^{-i\la \lambda, \xi\ra_{\R^d}}$, one can represent it as
$$X(\lambda)= \int_{\R^d} \zeta(\lambda,\,\xi) W(\De \xi),$$
which means that $(X,\phi)$ for $\phi\in \mathcal S(\R^d)$ has the stochastic integral representation
$$(X,\phi)= \int_{\R^d} \widehat{\phi }(\xi) W(\De \xi).$$
It is well-known (\citet[Chapter 1]{Lifshits}) that for any $f\in L^2_{\mathbb C}(\R^d)$ the integral above is well-posed.
Under the control measure $\mu$, the isometry property of the stochastic integral gives us the covariance of the field as~\eqref{eq:cov}. Note that since $W$ is a complex white noise with control measure $\mu$,  the field can also be represented by using a standard complex white noise $\widetilde W$ (with control measure $\De \xi$) in such a way that 
$$X(\lambda)= \int_{\R^d} \zeta(\lambda,\,\xi)\la \xi \ra_m^{-d} \widetilde W(\De \xi).$$
The above white noise representation helps to create the first example of white-noise cut-off. Pick now an arbitrary $\eps>0$. We denote the white noise cut-off as
\eq{}\label{eq:X_eps_I}
{X}_{{\epsilon}}{(x)}:=\frac1{\omega_d}\int_{B(0,\,1/\eps)} \zeta(x,\,\xi) W(\De\xi).\eeq{}
Here $\omega_d= 2\pi^{d/2}/\Gamma(d/2)$ is the volume of the $d$-dimensional unit ball $\mathbb S^d$. Such cut-offs are also known as ultra-violet (UV) cut-offs (see \citet{RhoVarRev}). 
We call this a cut-off for the field since if we denote by 
\eq{}
\label{eq:K_eps}
K_\eps(x,y)= \E{X_\eps(x)X_\eps(y)}=\int_{B(0,\,1/\eps)} \zeta(x-y,\,\xi) \la\xi\ra_m^{-d}\De\xi\eeq{} 
then for $f, g$ compactly supported smooth functions one has
$$\lim_{\eps\to 0}\int_{\R^d}\int_{\R^d} f(x)g(y) K_\eps(x,y)\De x \De y= \int \widehat f(\xi) \overline{\widehat{g(\xi)}}\la\xi \ra_m^{-d} \De \xi,$$
and the right hand is the same as~\eqref{eq:cov}.
One can also introduce many other cut-offs. One example is by taking a mollifier $\theta(\cdot)$ that satisfies 
\begin{enumerate}
 \item $\theta$ is positive definite and symmetric,
\item $\int_{\mathbb R^d}\theta(x)\De x=1$,
\item $|\theta(x)|\le \frac1{1+|x|^{d+\gamma}}$ for some $\gamma>0$,
\end{enumerate}
(an example is the Gaussian density). One can define $\theta_\eps(\cdot):=\eps^{-2}\theta\left(\eps^{-1}\left(\cdot\right)\right)$, so that the process $X_\eps(x):=\left(X,\,\theta_\eps\left(x-\cdot\right)\right)$ satisfies
\eq{}\label{eq:referee}
 {X_\eps(x)}=\int_{\R^d}\zeta(x-y,\,\xi)\widehat\theta(\eps \xi) W(\De \xi)
\eeq{}
where $W$ is again a complex white noise with control measure $\mu(\De \xi)=\la \xi \ra_m^{-d}\De \xi$. When $\theta$ is the normalised indicator function of the sphere $\mathbb S^2$ this cut-off is called {\em sphere average} (\citet{DS10, HMP})
\subsubsection{Integral cut-offs}
This cut-off has been extensively used by \citet{RhoVarRev} as it follows under the scope of the work of \citet{Kah85}. Consider the massive GFF on $\R^d$. For that one observes that $K_\eps(x,y) \to K_0(m\Vert x-y\Vert )$ as $\eps \to 0$ and $x\neq y$ (on the diagonal the modified Bessel function is infinite). For $x\neq y$ one may write
$$K_0(m\Vert x-y\Vert )= \int_1^{\infty} k_m(u\Vert x-y\Vert )\frac{\De u}{u}$$ where $$k_m(z)=\frac12\int_0^\infty e^{-\frac{m^2|z|^2}{2v}}e^{-v/2}\De v.$$
Now one denotes the {integral cut-off } of the covariance for $x,y\in \R^d$ as
\eq{}\label{eq:cov2}{H}_{{\epsilon}}{(x,y)}= \int_1^{1/\eps}k_m(u\Vert x-y\Vert )\frac{\De u}{u}\eeq{}
and associates to it a centered Gaussian process (we show in the appendix that $K_\eps$ gives rise to a positive definite functional). Note that even when $x=y$ this is well defined and it follows that $H_\eps(x,x)\sim -\log \eps$ as ${\eps\to 0}$. For the planar GFF one can define the integral cut-offs as follows. One considers for $\eps>0$
\begin{equation}\label{gff:icutoff}
 G_{\eps, D}(x,y)=2\pi \int_{\eps}^{+\infty} p_D(s,x,y) \De s.
\end{equation}
It is well-known that $G_{\eps, D}$ is a positive definite kernel (for a proof see \citet[Section 5.2]{RhoVarRev}) and hence one can consider a Gaussian process $X_\eps(x)$ such that
$\E{X_\eps(x)X_\eps(y)}= G_{\eps, D}(x,y)$. Note that, as before, it follows that for $f, g\in C_0^\infty(D)$ one has that
$$\lim_{\eps\to 0}\int_D\int_D f(x)g(y) G_{\eps, D}(x,y) \De x \De y =\la f, \Delta^{-1}g\ra_{L^2}.$$

We observe that $H_\eps(x,y)$ in \eqref{eq:cov2} and $K_\eps(x,y)$ in \eqref{eq:K_eps} are different pointwise and indeed it can be shown that there are $x,\, y\in \R^d$, such that $K_\eps(x,y)$ 
takes negative values whilst $H_\eps(x,y)$ is always positive. 
\subsubsection{Transition semigroup cut-offs}
These approximations rely on the particular transition semigroup of the massless Gaussian field, which is given in terms of the transition kernel of Brownian motion and follow somehow a mixed approach between the integral cut-off (compare for example \eqref{gff:icutoff} and \eqref{eq:cov2}) and the white-noise integration. Let us start with the planar case to illustrate the technique.
Let $W$ be a standard space-time Gaussian white noise on $D\times (0,\infty)$ with the Lebesgue measure as control measure; define the stochastic integral corresponding to the Gaussian free field as
$$X(x)= \sqrt{2\pi}\int_{D\times (0,\infty)} p_D(s/2, x,y ) W(\De y,\De s).$$
Now one represents the approximating field as
$$X_\eps(x):=\sqrt{2\pi}\int_{D\times (\eps,\infty)} p_D(s/2, x,y ) W(\De y,\De s).$$
It follows again that $\E{X_\eps(x) X_\eps(y)}= 2\pi\int_{\eps}^\infty p_D(s,x,y) \De s$ (see \citet{RhoVarRev}).

The very same decomposition works for the massive GFF too. Let $p(t,\,x,\,y)$ be the transition kernel for standard Brownian motion. Knowing that
$$
B_1^{-d}u(x)=\frac{\sqrt{2\pi}}{\Gamma(d/2)}\int_0^{+\infty} \int_{\R^d}\e^{t}t^{d/2-1}p(t,\,x,\,y)u(y)\De t\De y
$$
we can set
$$
{X}_{{\epsilon}}{(x)}:=\frac{\sqrt{2\pi}}{\sqrt{\Gamma(d/2)}}\int_{\R^d\times [\eps,\,+\infty)}\e^{t/2}t^{\frac12\left(d/2-1\right)}p(t/2,\,x,\,y)W(\De t,\,\De y).
$$
This decomposition extends in general to 
any operator whose action can be represented through the Brownian motion semigroup 
(as for example \citet{HuZae09}). Being very similar to integral cut-offs such as \eqref{gff:icutoff}, in the paper we do not treat these approximations as separate cases but refer to integral cut-offs for more general properties.

\subsection{Main results}\label{sec:main}
In this section we discuss the main results stated in the previous section. We give some general sufficient conditions under which the lower bound and upper bound on the fractal dimension can be proved. We also give sufficient conditions for comparing thick points of two different cut-offs. Later in the article we show that these sufficient conditions are satisfied by almost all of the the cut-offs described above. 

Notation: we write $a\wedge b:=\min\left\{a,\,b\right\}$ and $f(t)\sim_{t\to 0} g(t)$ means that $f(t)/g(t)\to 1$ as $t\to 0$.

\begin{theorem}\label{thm:upper}
 If $\left(X_\eps(x)\right)_{\eps\geq 0,\,x\in \R^d}$, $d\geq 2$, is a centered Gaussian process satisfying
\begin{enumerate}
\item[\label{(A)}{(A)}] for all $R>0$ and for all $x,y \in B(0,R)$ and $\eps, \eta\ge 0$ we have
$$\E{(X_\eps(x)-X_\eta(y))^2}\leq \frac{\Vert x-y\Vert +|\eta-\eps|}{\eta\wedge \eps},$$
\item[\label{(B)}{(B)}] the variance of the process satisfies
$$G(\eps):=\E{X_\eps(x)^2} \sim_{\eps\to 0} -\log \eps.$$
\end{enumerate}
Then letting $$T_\ge(a, R)=\left\{x\in B(0,R):\, \lim_{\eps\to 0}\frac{X_{\eps}(x)}{G(\eps)}\ge a\right\}$$ we have
 for $a\le \sqrt{2d}$ that $\dime_H( T_\ge(a, R)) \leq d-{a^2}/{2}$  almost surely, and for $a>\sqrt{2d}$ that $T_{\ge}(a, R)$ is empty almost surely.
\end{theorem}

Theorem~\ref{thm:upper} is stated for balls of radius $R$, but it can be used to derive the upper bound by first covering the space with a countable number of balls and then using the countable stability property of the Hausdorff dimension, which reads as
\eq{}\label{eq:stability}
\dim_H\left(\bigcup_{n\in \N}B_n\right)=\sup_{n\in \N}\dim_H\left( B_n\right)
\eeq{}
for an arbitrary collection of sets $(B_n)_{n\in\N}$.
In the following Corollary we treat as a separate case the Massless GFF, both for its importance and for the slightly different proof. 
\begin{corollary}\label{corol:GFF}
Let $D$ be a bounded, convex regular domain. For $\delta>0$, denote $D^{(\delta)}:=\{x\in D: \mathrm d(x, \partial D)>\delta)\}$. If $X_\eps(x)$ is a planar (massless) Gaussian free field integral cut-off, satisfying assumptions \hyperref[(A)]{(A)} and \hyperref[(B)]{(B)} on $D^{(\delta)}$ for any $\delta>0$. 
Then the conclusion of Theorem \ref{thm:upper} holds with $d=2$.
\end{corollary}
In Section~\ref{examples} we will see that most of the cut-offs discussed in Subsection~\ref{subsec:cut-description} satisfy the assumptions of Theorem~\ref{thm:upper} and Corollary~\ref{corol:GFF}. A brief sketch of the proof is as follows. The condition \hyperref[(A)]{(A)} allows us to have a modification which has nice bounds on the spatial and time variable almost surely. We use a strong version of Kolmogorov-Centsov theorem from \citet{HMP} to derive this. Using these path properties it is possible to get an explicit cover of the space and also get good bounds for the diameters of the sets used to form the cover. The upper bound then follows easily from the definition of Hausdorff dimension.

Now we give some sufficient conditions on the cut-off for which we have a matching lower bound. We state the results for discrete time for ease of exposition, noting that it can be extended to continuous time if the processes have a continuous modification.

\begin{theorem}\label{thm:lower}
Let $\{X_n(x), x\in D, \, n\ge 1\}$ be a continuous centered Gaussian process with covariance kernel $q_n(x,y)$ which satisfies the following properties:
\begin{enumerate}

\item[\label{(C)}{(C)}] there exists a uniformly bounded function $H_U\colon D\times D\to \R$ such that for all $n\ge 1$ and $x\neq y$,
\begin{equation}\label{condC}
q_n(x,y) \le \log\frac1{\|x-y\|}+H_U(x,y),
\end{equation}
and there exist a constant $C'$ such that, for all $N\ge 1$, there exist $k_0\ge 1$ such that whenever $\|x-y\|\le \e^{-N}$, one has
\begin{equation}\label{eq:wn2}
q_k(x,\,y)-q_N(x,\,y) \le \log\frac1{\|x-y\|}-N +C'\quad \text{ for all $k\ge k_0$}.
\end{equation}
\item[\label{(D)}{(D)}] There exists a sequence of positive definite covariance kernels $\{p_k(x,\,y)\}_{k\ge1}$ such that $q_n(x,\,y)=\sum_{k=1}^n p_k(x,\,y)$ and $p_k(x,\,x)\le c_k$ for all $x,\,y\in D$, $n,\,k\ge 1$ and $\sum_{k\ge 1}c_k/k^2<+\infty$ and $q_n(x,\,x)\sim_{n\to+ \infty} n$.
\end{enumerate}
Consider the set of thick points 
$$T(a)=\left\{x\in D\colon \lim_{n\to+\infty} \frac{X_n(x)}{n}=a\right\}.$$
Then $\dime_H(T(a))\ge d-{a^2}/{2}$ almost surely.
\end{theorem}

\begin{remark}\label{remark:improvement}
Although the conditions look technical it is easy to see that the above assumptions \hyperref[(C)]{(C)} and \hyperref[(D)]{(D)} get satisfied when the following two conditions are assumed.
\begin{enumerate}
\item[\label{(C1)}{(C1)}] There exist uniformly bounded functions $H_U$ and $H_L$ such that for $x\neq y$  \begin{equation}
\log\frac1{\|x-y\|}-H_L(x,y) \le q_n(x,y)\le \log\frac1{\|x-y\|}+ H_U(x,y)
\end{equation}
\item[\label{(D1)}{(D1)}]  There exists a sequence of positive definite covariance kernels $\{p_k(x,\,y)\}_{k\ge1}$ such that $q_n(x,\,y)=\sum_{k=1}^n p_k(x,\,y)$ and $p_k(x,\,x)=1$ for all $x,\,y\in D$, $n,\,k\ge 1$.
\end{enumerate}
\end{remark}


The proof is essentially given in \citet{Kah85}. The condition \hyperref[(D)]{(D)} allows one to construct a positive martingale using measures of the form \eqref{eq:lqg} which converge for every bounded set $A$. It is then standard to construct a limiting a measure out of it. We briefly sketch the idea of Kahane to show how one can easily adapt it to the above general conditions.  

Naturally one could ask oneself whether one can compare covariances of two cut-offs to deduce the behavior of thick points. Our next result is in that direction.

\begin{theorem}\label{thm:comparison}

Let $X_\eps(x)$ and $\widetilde X_\eps(x)$ be two cut-off families for the same field on $D$. Let $T(X,a)$ and $T(\widetilde{X},a)$ be the set of $a$-thick points for $X_\eps(x)$ and $\widetilde{X}_\eps(x)$ respectively. Call $Z_\eps(x):=X_\eps(x)- \widetilde X_\eps(x)$. Suppose $Z_\eps(x)$ satisfies the following assumption:
\begin{itemize}
 \item[\label{(E)}(E)] $Z_\eps(x)$ is symmetric in $x$ and there exists universal constants $C>0$, $C'>0$ independent of $\eps$ and $x$ such that
\eq{}\label{eq:ineq1}
\E{Z_\eps(x)^2}\le C \eeq{}
and
\eq{}\label{eq:ineq2}\E{(Z_\eps(x)-Z_\eps(y))^2}\leq C'\frac{\norm{x-y}}{\eps}.\eeq{}
\end{itemize}
Then for all $a>0$ we have $\mathrm{dim}_H(T(X,a))= \mathrm{dim}_H(T(\widetilde{X},\,a))$ almost surely.
\end{theorem}

Again we will give an example in Section~\ref{examples} where condition \hyperref[(E)]{(E)} is satisfied. 

To prove Theorem~\ref{thm:comparison} we show first that using Sudakov-Fernique's inequality one can compare the maxima of the Gaussian process $Z_\eps(x)$ with a multivariate version of the 
Ornstein-Uhlenbeck process for which the order of expected maxima can be easily derived. To pass to the almost sure version one uses bounded variances and Borell's inequality. This allows one to compare the set of thick points and derive the final result.

\section{Examples}\label{examples}
In this Section we explicitly show cut-offs that satisfy the assumptions of our 
theorems. We will concentrate on massive and massless GFFs but the results in general can be applied to centered Gaussian process with appropriate covariance structure too.   
\subsection{White noise cut-off for massive GFF \eqref{eq:X_eps_I}}
We will show \hyperref[(A)]{(A)}-\hyperref[(C)]{(C)}-\hyperref[(D)]{(D)} (\hyperref[(B)]{(B)} is a standard computation). The latter two conditions are outlined in the lecture notes by \citet{RVNotes} and hence we briefly sketch their proof.
\begin{enumerate} 
\item[{\hyperref[(A)]{(A)}}] Let us without loss of generality assume $\eps_1<\eps_2$ then it follows that
\begin{eqnarray*}
&&\E{(X_{\eps_1}(x)- X_{{\eps_2}}(y))^2}=\int_{\R^d}
 \frac{(\zeta(x,\xi)\one_{B(0,1/{\eps_1})}-\zeta(y,\xi)\one_{B(0,1/{\eps_2})})^2}{(\norm{\xi}^2+m^2)^{d/2}}\De\xi\\
&&=\int_{B(0,1/{\eps_2})}
 \left(\zeta(x,\xi)-\zeta(y,\xi)\right)^2\frac1{(\norm{\xi}^2+m^2)^{d/2}} \De\xi\\
&&+\int_{1/{\eps_2}< \|\xi\|\le
 1/{\eps_1}}
 \frac{\zeta(x,\xi)^2}{(\norm{\xi}^2+m^2)^{d/2}} \De\xi \leq C\frac{\norm{x-y}}{\eps_2}+\frac{|\eps_2-\eps_1|}{\eps_1}\leq C\frac{\norm{x-y}+|\eps_2-\eps_1|}{\eps_1\wedge \eps_2}
\end{eqnarray*}
where we have used the inequality $\left\lvert \log\left(\frac{x}{y}\right)\right\rvert\leq \lvert x-y\rvert/ {x\wedge y}$.
\item[{\hyperref[(C)]{(C)}}] The required Gaussian process in this case is taken to be $X_{\e^{-n}}(x)$, so that
$p_k(x,\,x)= \E{ X_{\e^{-k}}(x) X_{\e^{-k}}(y)}$. Recall that $G(\eps)= \omega_d^{-1}\int_0^{1/\eps} t^{d-1}\left(m^2+t^2\right)^{d/2}\De t$ and hence it follows that $q_n(x,\,x)=G(\e^{-n})\sim_{n\to +\infty} n$. Set $c_k:=\e^k$. Denoting by $\phi(u)=\left(1+u^d\right)\left(m^2+u^2\right)^{-d/2}$ one has
\eqa{*}
&&q_N(x,\,y)=\frac{1}{\omega_d}\int_{0}^{c_N}\left(\int_{\mathbb S^d}\e^{\i\la x-y,\,r\vec s\ra}\phi(r\vec s)\De\vec s\right)\frac{r^{d-1}}{1+r^d}\De r \\
&&=\frac{1}{\omega_d}\int_{0}^{c_N\wedge \|x-y\|^{-1}}\left(\int_{\mathbb S^d}\e^{\i\la x-y,\,r\vec s\ra}\phi(r\vec s)\De\vec s-1\right)\frac{r^{d-1}}{1+r^d}\De r +\int_{c_N\wedge \|x-y\|^{-1}}^{c_N}\frac{r^{d-1}}{1+r^d}\De r\\
&&+\frac{1}{\omega_d}\int_{c_N\wedge \|x-y\|^{-1}}^{c_N}\left(\int_{\mathbb S^d}\e^{\i\la x-y,\,r\vec s\ra}\phi(r\vec s)\De\vec s\right) \frac{r^{d-1}}{1+r^d}\De r 
\eeqa{*}
Call $H_r(x):=\frac{1}{\omega_d}\int_{\mathbb S^d}\e^{\i\la x,\,r\vec s\ra}\phi(r\vec s)\De\vec s.$ Now note that one break down $H_r(x)$ as
$$H_r(x)=\frac{1}{\omega_d}\left(\int_{\mathbb S^d}\e^{i\la x-y,\,r\vec s\ra}\left(\phi(r\vec s)-1\right)\De\vec s+\int_{\mathbb S^d}\cos(r\|x\|\vec s)\De\vec s\right)$$ 
and now using the following inequalities from~\citet{RVNotes} the condition follows 
\begin{equation}\label{eq:quadrupl}
\left| H_r(x)-1\right|\le \frac{C}{\left(1+r\right)^\alpha}+r|x|, \quad \left|\int_{\mathbb S^d}\cos\left(r\|x\|\vec s\right)\De\vec s\right|\le\frac{C}{\left(1+r\|x\|\right)^\eta},
\end{equation}
where $\alpha>0$ and $\eta \in \left(0,\,\frac12\right)$ and $C$ is some generic positive constant.
\item[{\hyperref[(D)]{(D)}}] We write as before
\eqa{*}
&&q_k(x,\,y)-q_N(x,\,y)=\frac{1}{\omega_d}\int_{c_N}^{c_k}\left(\int_{\mathbb S^d}\e^{i\la x-y,\,r\vec s\ra}\phi(r\vec s)\De\vec s\right)\frac{r^{d-1}}{1+r^d}\De r \\
&&=\frac{1}{\omega_d}\int_{c_N}^{\|x-y\|^{-1}}\left(\int_{\mathbb S^d}\e^{i\la x-y,\,r\vec s\ra}\phi(r\vec s)\De\vec s-1\right)\frac{r^{d-1}}{1+r^d}\De r +\int_{c_N}^{\|x-y\|^{-1}}\frac{r^{d-1}}{1+r^d}\De r\\
&&+\frac{1}{\omega_d}\int_{\|x-y\|^{-1}}^{c_k}\left(\int_{\mathbb S^d}\e^{i\la x-y,\,r\vec s\ra}\phi(r\vec s)\De\vec s\right) \frac{r^{d-1}}{1+r^d}\De r.
\eeqa{*}
Before proceeding note that the break-up is possible since $\|x-y\|^{-1}\ge \e^N=c_N$ and we take $k$ large enough so that $c_k \ge\|x-y\|^{-1} $ eventually. The first and third term are bounded uniformly in $x$ and $y$ by \eqref{eq:quadrupl}, whereas the main contribution comes from 
\eqa{*}
&&\int_{c_N}^{\|x-y\|^{-1}}\frac{r^{d-1}}{1+r^d}\De r\le \log\|x-y\|^{-1}-N+C(d).
\eeqa{*}
Note that $p_k(x,x)\le 1$ and $q_n(x,x)\sim n$, and this gives the desired result.
\end{enumerate}
\subsection{Integral cut-off for massive GFF \eqref{eq:cov2}}
We will briefly point out the computation for \hyperref[(A)]{(A)}. 
\begin{enumerate}
\item[{\hyperref[(A)]{(A)}}] We have
$$
\E{\left(X_\eps(x)-X_\eps(y)\right)^2} \leq  \vert\E{X_\eps(x)^2-X_\eps(x))X_\eps(y)} \rvert+\vert\E{X_\eps(y)^2-X_\eps(x)X_\eps(y)} \rvert. $$
We look at the first term and the other follows similarly. Using an appropriate change of variables we can write the first term as sum of two terms in the following way
\eqa{*}
&&\left|\E{X_\eps(x)^2-X_\eps(x)X_\eps(y)} \right|=\frac12\int_1^{1/\eps}\int_0^{\infty} \e^{-v/2}\left(1-\e^{-\frac{\Vert x-y\Vert ^2 m^2 u^2}{2v}}\right)\De v\frac{\De u}{u}\\
&&=\frac12 \int_0^{\frac{\Vert x-y\Vert  m}{\eps}} \left(\e^{-\frac{\Vert x-y\Vert ^2 m^2}{2s}}-\e^{-\frac{\Vert x-y\Vert ^2 m^2}{2\eps^2 s}}\right)\left(1-\e^{-s/2}\right)\frac{\De s}{s}\\
&&+\int_{\frac{\Vert x-y\Vert  m}{\eps}}^{\infty} \left(\e^{-\frac{\Vert x-y\Vert ^2 m^2}{2s}}-\e^{-\frac{\Vert x-y\Vert ^2 m^2}{2\eps^2 s}}\right)\left(1-\e^{-s/2}\right)\frac{\De s}{s}.
\eeqa{*}
Now in the first integral the integrand is bounded in absolute value by $C v/2$, hence the whole integral is smaller than $C{\Vert x-y\Vert  m}/{\eps}$. As for the second integral note that using the inequality $1-\e^{-x}\le x$ it follows that the integrand is bounded by $\|x-y\|^2m^2(\eps^{-2}-1) s^{-2}$ and hence again after integrating the integral cannot be larger than $C\|x-y\|m/\eps$.
Here for \hyperref[(C)]{(C)}-\hyperref[(D)]{(D)} one can take 
$$p_k(x,y)=\int_{\e^{(k-1)}}^{\e^{k}} k_m(u(x-y)) \frac{\De u}{u},\quad k\in \N$$ and it is straightforward to verify the conditions.
\end{enumerate}

\subsection{Planar GFF semigroup cut-off \eqref{gff:icutoff}}
The proof for the planar GFF is a bit more involved than for 
other cut-offs, and requires some preliminary lemmas and notations. We also would like to remind here that a proof tailored on the 2-d massless GFF for Theorem~\ref{thm:upper} is given in Corollary~\ref{corol:GFF}. We first show conditions \hyperref[(A)]{(A)} and \hyperref[(B)]{(B)} and for future references we put it as lemma.
\begin{lemma}\label{lem:condAB:gff}
Fix $\delta>0$, and for a set $D$ assume that $D^{(\delta)}$ is a convex bounded domain. There exists a constant $C=C (\delta)$ such that
$$\E{(X_\eps(x)-X_\eta(y))^2}\leq C\frac{\Vert x-y\Vert +|\eta-\eps|}{\eta\wedge \eps}$$
holds for all $x,y\in D^{(\delta)}$ and $G(\eps)= \var{X_\eps(x)}\sim_{\eps\to 0} -\log\eps$.
\end{lemma}

\begin{proof}
We first begin by showing the second statement. Recall that
$$G(\eps)=2\pi\int_{\eps}^\infty p_D(t,x,x)\De t.$$
Also note that from \citet[Section 2.4]{Law05} we have the following upper and lower bounds on $p_D(t,x,x)$,
\begin{equation}\label{eq:Law_1}
\frac{1}{2\pi t}-\frac1{\pi \e
\left(\mathrm{d}(x,\,\partial D)\right)^2}\le p_D(t,x,x)\le \frac{1}{2\pi t}.
\end{equation}

Fix $t_0>1$, then we ignore the part from $(t_0,\infty)$ by using \citet[Lemma 2.28]{Law05}, since
$$\int_{t_0}^\infty p_D(t,x,x) \De t\le C(x,\delta)\int_{t_0}^\infty 
\frac1{t(\log t)^2}\De t < \infty.$$
Now using the fact that $x\in D^{(\delta)}$, it follows from~\eqref{eq:Law_1} that $G(\eps)\sim -\log \eps$ as $\eps\to 0$.
At this point we show the first bound. The case $x=y$ and $\eps<\eps'$ is easier using the fact that $p_D(t,x,y)\le p(t,x,y)$ and the properties of the heat kernel. We therefore concentrate on showing
%
Condition \hyperref[(A)]{(A)} for $x\neq y$, we use the representation from  \citet{MoerPer} 
$$p_D(t,x,y)=p(t,x,y)-\mathrm E_x\left[ p(t-T_D, B_{T_D},y)\one_{\left\{T_D<t\right\}}\right]$$
for $\mathrm E_x$ the law of a standard Brownian motion $B$ with $B_0=x$. Note that
\begin{align*}
\E{(X_\eps(x)- X_{\eps}(y))^2} &\le \left|\int_{\eps}^\infty \left(p_D(t,x,x)-p_D(t,x,y)\right)\De t\right| \\
&+\left|\int_\eps^\infty \left(p_D(t,y,y)-p_D(t,x,y)\right)\De t\right|.
\end{align*}
We shall show that
\eq{}\left |\int_\eps^\infty \left(p_D(t,x,x)-p_D(t,x,y)\right)\De t\right| \le C \frac{\Vert x-y\Vert }{\eps}.\label{eq:bound_prob}\eeq{}
The other part follows similarly. So now note that \eqref{eq:bound_prob} is satisfied if one replaces $p_D$ with $p$. We take $D$ to be a bounded domain, hence there exists $M>0$ such that $\Vert x-y\Vert\le M$ for all $x,\,y\in D^{(\delta)}$. So
\begin{align}
\int_{\eps}^\infty \left(p_D(t,x,x)-p_D(t,x,y)\right)\De t& =\int_\eps^\infty \left(1-\exp\left(-\frac{\Vert x-y\Vert }{2t}\right)\right )\frac{\De t}{t}\nonumber \\
&\le \frac{\Vert x-y\Vert ^2}{2}\int_\eps^\infty \frac{\De t}{t^2}\le\frac{M}{2} \frac{\Vert x-y\Vert }{\eps}.\label{eq: epsbound}
\end{align}
Now we need to show the term containing the expectation has a similar bound. First recall that using a multivariate version of the mean value theorem one has
$$\left|p(t, z, x)-p(t, z, y)\right|\le \left|\nabla p(t,z,(1-\lambda)x+\lambda y)\right|\norm{x-y}$$ with $\lambda\in [0,\,1]$. We use then the notation $\xi:=(1-\lambda)x+\lambda y$ to denote a point on the segment starting at $x$ and ending at $y$. Observe that $\xi\in D^{(\delta)}$.  From \citet{Sal-Cos} we have for any $\kappa\in (0,\,1)$
\eq{}\label{eq:S-C}
\Vert \nabla_{\xi} p(t,z, \xi)\Vert\leq \frac{C(\kappa)}{\sqrt t V(z,\,\sqrt t)}\exp\left(-\frac{\norm{z-\xi}^2}{4(1-\kappa)t}\right)
\eeq{}
and $V(x,\,r)$ is the volume of $B(x,\,r)$. Now using this inequality we have

\begin{align*}
&\int_\eps^\infty \mathrm E_x\left[ \left(p(t-T_D, B_{T_D}, x)-p(t-T_D, B_{T_D}, y)\right)\one_{\left\{ t>T_D\right\}}\right]\De t\\
&=\int_0^\infty \mathrm E_x\left[ \left(p(t-T_D, B_{T_D}, x)-p(t-T_D, B_{T_D}, y)\right)\one_{\left\{ t>T_D\vee \eps \right\}}\right]\\
&\stackrel{t-\tau_D=:s}{\leq} \mathrm E_x\left[\int_{0}^\infty  \left|p(s, B_0, x)- p(s, B_{T_D}, y)\right|\De s\right]\\
&\stackrel{\eqref{eq:S-C}}{\leq}\Vert x-y\Vert \mathrm E_x\left[\int_{0}^\infty  \left|p(s, B_0), x)- p(s, B_{T_D}, y)\right|\De s\right]\\
&\le \Vert x-y\Vert \mathrm E_x\left[\int_{0}^\infty  \frac{\exp\left(-\frac{-\Vert B_{T_D}-\xi \Vert^2}{4(1-\kappa)s}\right)}{s^{3/2}}\De s\right]\le \Vert x-y\Vert \mathrm E_x\left[\int_{0}^\infty  \frac{\exp\left(-\frac{-\delta^2}{4(1-\kappa)s}\right)}{s^{3/2}}\De s\right] \\
&=C(\delta,\,\kappa)\Vert x-y\Vert\le  C(\delta,\,\kappa)\frac{\Vert x-y\Vert}{\eps}.
\end{align*}
Here we have used the fact that $B_{T_D}\in \partial D$ and since $\xi\in D^{(\delta)}$ we have that $\Vert B_{T_D}-\xi \Vert \ge \delta$. So the above inequality combined with \eqref{eq: epsbound} shows~\eqref{eq:bound_prob} and hence completes the proof of the Lemma.
\end{proof}

\begin{enumerate}
\item[{\hyperref[(C)]{(C)}}] One has $$q_n(x,y)=2\pi\int_{e^{-n}}^\infty p_D(t,x,y)\De t, \qquad p_k(x,\,y)=2\pi\int_{\e^{-(k+1)}}^{\e^{-k}}p_D(t,x,y)\De t.$$ Using $p_D(t,x,y)\le p(t,x,y)$ the first bound in \hyperref[(C)]{(C)} easily follows. For~\eqref{eq:wn2} note that $$q_k(x,y)-q_N(x,y)=2\pi\int_{\e^{-k}}^{\e^{-N}}p_D(t,x,y)\De t$$ and since we consider $\|x-y\|\le e^{-N}$ one breaks this integral into two further integrals over $[\e^{-k}, \|x-y\|)$ and $[\|x-y\|, \e^{-N})$. By $p_D(t,x,y) \le \left(2\pi t\right)^{-1}$ one gets that the first integral is bounded as $x,\,y$ belong to a bounded domain, and the second by $-\log\|x-y\|-N$. Hence the bound follows.  

\item[(D)] Using $p_D(t,x,x)\le \left(2\pi t\right)^{-1}$ we obtain that $p_k(x,x)\le 1$. It follows from~\eqref{eq:Law_1} that $q_n(x,x)\sim_{n\to +\infty }n$. 

\end{enumerate}

\subsection{Example for comparison: cut-offs \eqref{eq:X_eps_I} and \eqref{eq:referee}}\label{subsec:example_comp}
This example illustrates the fact that the effect of the mollifier $\theta$ in \eqref{eq:referee} does not affect the structure of thick points, as one might rightly expect. Indeed let us show that \eqref{eq:ineq1} holds. 
Set $Z_\eps(x):=X_\eps(x)-\widetilde X_\eps(x)$ for $X_\eps(x)$ of \eqref{eq:X_eps_I} and $\widetilde X_\eps(x)$ of \eqref{eq:referee}. We have 
\begin{eqnarray}
&&\E{\left(X_\eps(x)-\widetilde X_\eps(x))\right)^2}\le 
\int_{B(0,\,1/\eps)}\frac{\left| 1-\widehat\theta(\eps\xi)\right|^2}{\left(\norm{\xi}^2+m^2\right)^{d/2} }\De \xi+\int_{B(0,\,1/\eps)^\mathrm c}\frac{\left| \widehat\theta(\eps\xi)\right|^2}{\left(\norm{\xi}^2+m^2\right)^{d/2} }\De \xi\nonumber\\
&&\leq \int_{B(0,\,1)}\frac{\left| 1-\widehat\theta(\xi)\right|^2}{\left(\norm{\xi}^2+\eps^2 m^2\right)^{d/2} }\De \xi+\int_{B(0,\,1)^\mathrm c}\frac{\left| \widehat\theta(\xi)\right|^2}{\left(\norm{\xi}^2+\eps^2 m^2\right)^{d/2} }\De \xi\nonumber\\
&&\le \int_{B(0,\,1)}\frac{\left| 1-\widehat\theta(\xi)\right|^2}{\norm{\xi}^d }\De \xi+\int_{B(0,\,1)^\mathrm c}\frac{\left| \widehat\theta(\xi)\right|^2}{\norm{\xi}^d }\De \xi\le C.\label{eq:X}
\end{eqnarray}
To obtain \eqref{eq:ineq2} observe that
\begin{eqnarray}
&&\mathbb E\left[\left(Z_\eps(x)-Z_\eps(y)\right)^2\right]\le 
\int_{B(0,\,1/\eps)}
\frac{(\zeta(x,\xi)-\zeta(y,\xi))^2\left(1-\hat\theta(\eps\xi)\right)^2}{(
\norm{\xi}^2+m^2)^{d/2} }\De \xi\nonumber
\end{eqnarray}
and from here one can proceed starting over again as in \eqref{eq:X} to conclude the proof of the condition. It is interesting to note when $\theta$ is the indicator of the sphere, one can derive the sphere average process from the white noise cut-off and visa-versa.

\section{Proof of the main results}\label{sec:proof}
Throughout this section $C$, $C'>0$ will represent generic constants which may differ from line to line.
\subsection{Proof of Theorems~\ref{thm:upper} and ~\ref{thm:lower}}
\begin{proof}[Proof of Theorem~\ref{thm:upper}]
First, we claim that by Assumption \hyperref[(A)]{(A)} of Theorem~\ref{thm:upper}, there exists a modification $\widetilde X_\eps(x)$ of
$X_\eps(x)$ such that for every $\gamma\in (0, 1/2)$ and $\chi, \zeta>0$ there exists $M>0$ such that
\begin{equation}\label{eq:KC}
\left|\widetilde X_{\eps_1}(x) - \widetilde X_{\eps_2}(y)\right| \le M \left(\log \frac1{\eps_2}\right)^\zeta\frac{\left(|(x,\eps_1)-(y,\eps_2)| \right)^\gamma}{\eps_2^{(1+\chi)\gamma}}
\end{equation}
for all $x, y\in B(0, R)$ and $\eps_1, \eps_2\in (0,1]$ and $\eps_2/\eps_1\in (1/2,2]$.
Indeed, by \hyperref[(A)]{(A)} we have that 

$$\E{\left(X_{\eps_1}(x)-X_{\epsilon_2}(y)\right)^\alpha}\le C \left(\frac{\|{x-y}\|+\abs{\epsilon_1-\epsilon_2}}{\epsilon_1\wedge \epsilon_2}\right)^{\alpha/2}.
$$
We can find $\alpha$ and $\beta$ large enough such that $\abs{\f{\beta}{\alpha}-\f{1}{2}}<\delta$, and consequently by \citet[Lemma C.1]{HMP} there exists a modification $\widetilde X_\eps(x)$
a.s.\ for which~\eqref{eq:KC} holds. Without loss of generality we now work with this modification and with a slight abuse of notation denote it by $X_\eps(x)$.
Now we choose some suitable parameters according to the regularity condition above. Let $\chi>0 $, $r\in\left(0,\,\frac12\right)$, $\zeta\in(0,\,1)$, $\widetilde r=(1+\chi)r$, $K=\chi^{-1}$, $r_n=n^{-K}$, and
$$U_R:=\left\{x\in B(0, \,R):\,\lim_{n\to+\infty}\frac{X_{r_n}(x)}{G(r_n)}\geq a\right\}.$$ 
Since for $t\in (r_{n+1},r_n)$ we have by~\eqref{eq:KC} and the fact that $G(r_n)=C\log n(1+\o{1})$,
$$
\left|\frac{X_t(x)-X_{r_n}(x)}{ G(r_n)}\right|=\O{\frac{(\log n)^\zeta}{G(r_n)}}=\o{1}.
$$
This shows that   $T_{\ge}(a, R)\subseteq U_R$. Let $(x_{nj})_{j=1}^{\bar k_n}$ be an $r_n^{1+\chi}$-net for points in $B(0,R)$.
Denote
$$
\mathcal A_n:=\left\{j:\,\frac{X_{r_n}(x_{nj})}{G(r_n)}\geq a -\delta(n)\right\}
$$
with $\delta(n)=C(\log n)^{\zeta-1}$ (the constant $C$ can be adjusted accordingly). Again using~\eqref{eq:KC} it follows that,
 for all $N\geq 1$,  $\bigcup_{n\geq N}\bigcup_{j\in \mathcal A_N}B\left(x_{nj},r_n^{1+\chi}\right)$ covers $U_R$ with sets having maximal diameter $2 r_n^{1+\chi}$.

We first note the estimate $\prob{j\in \mathcal A_n}$ using the following Gaussian tail bound as follows:
\begin{eqnarray*}
&&\prob{j\in \mathcal A_n}\le \prob{\frac{X_{r_n}(x_{nj})}{\sqrt{ G(r_n)}} \ge (a-\delta(n) )\sqrt{G(r_n)}}\leq C (\log n)^{-1/2}n^{ -\frac{a^2}{2\chi}(1+\o{1})}
\end{eqnarray*}
Furthermore
\begin{equation}\label{upper:expectation}
\E{|\mathcal A_n|}\leq C(\log n)^{-1/2} \overline k_n r_n^{-d(1+\chi)}n^{ -\frac{a^2}{2\chi}(1+\o{1})}\leq (\log n)^{-1/2} n^{-\frac{a^2}{2\chi}+d+\frac{d}{\chi}+\o{1}}.
\end{equation}

By denoting $$\alpha= d-\frac{a^2}{2}+\chi\frac{d+\frac{a^2}{2}}{1+\chi},$$
we can estimate the size of the balls in the cover as follows: 
\begin{align*}
&\E{\sum_{n\ge N}\sum_{j \in \mathcal A_n}\diam(B(x_{nj},r_n^{1+\chi}))^\alpha}\le \sum_{n\ge N}(\log n)^{-1/2} r_n^{\alpha(1+\chi)}n^{-\frac{a^2}{2\chi}+d+\frac{d}{\chi}+\o{1}}\\
&\le \sum_{n\ge N}(\log n)^{-1/2}n^{-\frac1{\chi}(1+\chi)\alpha-\frac{a^2}{2\chi}+d+\frac{d}{\chi}+\o{1}}\le C\sum_{n\ge N}(\log n)^{-1/2} n^{-d}<+\infty.
\end{align*}
Therefore $\sum_{n\ge N}\sum_{j \in \mathcal A_n}\diam(B(x_{nj},r_n^{1+\chi}))^\alpha<+\infty$ a.s.\ and this implies  $\dime_{H}(T_\ge(a,R))\le d-\frac{a^2}{2}$ a.s.\ by letting $\chi \downarrow 0$.

Now we show that for every $R>1$, $T_\ge (a,R)$  is empty for $a^2> 2d$ using the above estimates. Since $a^2>2d$ we have that $\frac{a^2}{2\chi}-d(1+\frac1{\chi})>1$ and hence,

\begin{equation*}
\sum_{n\ge 1}\prob{\abs{\mathcal A_n}>1}\leq \sum_{n\ge 1}\E{\abs{\mathcal A_n}}\le \sum_{n\ge 1} n^{-\left(\frac{a^2}{2\chi}-d\left(1+\frac1{\chi}\right)\right)}<\infty.
\end{equation*}
and hence by the Borel-Cantelli lemma we can conclude that, if $\chi$ becomes arbitrarily small, $\abs{\mathcal A_n}=0$ eventually and so $T_\ge(a,R)$ is empty for $a^2>2d$ with probability one.
\end{proof}

\begin{proof}[Proof of Corollary~\ref{corol:GFF}]
Recall that from Lemma~\ref{lem:condAB:gff} we have that Conditions \hyperref[(A)]{(A)} and \hyperref[(B)]{(B)} hold for the restricted domain $D^{(\delta)}$, for any $\delta>0$. Now, since $D$ is a regular domain we can write
$$
D=\bigcup_{n\in \N}D^{(\delta_n)}
$$
for a suitable sequence $\delta_n\downarrow 0$ and $\mathrm{d}(D^{(\delta_n)}, D^c)>0$. 
Now that by repeating the arguments in the proof of Theorem~\ref{thm:upper} and using Lemma~\ref{lem:condAB:gff}, we get that for all $n$, $\dim_H\left(T(a, D^{(\delta_n)})\right)\le 2-{a^2}/{2}$ with probability one. Hence by \eqref{eq:stability}
we obtain that $\dim_H (T(a,D))\le 2-{a^2}/{2}$ almost surely. 
\end{proof}
Now we provide a proof of the lower bound. 
\begin{proof}[Kahane's proof of lower bound]
The proof is essentially Kahane's proof of non-degeneracy, so we only give a brief outline of it. The broad idea is to construct a measure $\nu$ giving full mass to $T(a)$ which has finite $\alpha$-energy for $\alpha<d-{a^2}/{2}$.\footnote{The $\alpha$th-energy of a measure $\nu$ is defined by $I_\alpha(\nu):=\int_{D\times D}\frac1{\|x-y\|^\alpha}\nu(\De x)\nu(\De y).$} To show this one looks at the class of measures $R_\alpha$ which is described as follows:
$R_\alpha$ is the set of all Radon measures $\sigma$ on $D$ such that for all $\eps>0$ there exist $\delta>0$, $C>0$, and a compact set $K_\eps\subseteq D$ with $\sigma(D\setminus K_\eps)<\eps$ such that $\sigma_K(\De x):=\one_{K_\eps}(x) \sigma(\De x)$ satisfies
$$\sigma_K( D) \le \mathrm{diam}( U)^{\alpha+\delta}\text{ for every open set $U\subseteq D$}.$$
It is easy to show that if $\mu\in R_\alpha$ then $I_\gamma(\mu)<\infty$ for all $\gamma<\alpha+\delta$.\\
{\bf Step 1 (martingale measures).} Consider $\sigma \in R_\alpha$ and define a random measure $Q^{(a)}\sigma$ in the following way: let, with a slight abuse of notation, $Q_n^{(a)}\sigma(\De x):= Q_n^{(a)}(x) \sigma(\De x)$ where 
$$Q_n^{(a)}(x)=\exp\left( aX_n(x)-\frac{a^2}{2} \E{X_n(x)^2}\right).$$
Note that due to condition $\hyperref[(D)]{\hyperref[(D)]{(D)}}$, there exist independent random variables $\left(Y_k(x)\right)$ with covariance kernel given by $p_k(x,\,y)$ and $X_n(x)=\sum_{k=1}^n Y_k(x)$ in law. It is almost immediate that for any Borel set $A$, $Q_n^{(a)}\sigma(A)$ is a positive martingale and hence converges almost surely. Now one can show that there exists a probability measure (denoted by $Q^{(a)}\sigma$) on $D$ such that $Q_n^{(a)}\sigma$ converges weakly almost surely to $Q^{(a)}\sigma$. Note that when $a^2< \alpha$, by the upper bound in condition $\hyperref[(C)]{\hyperref[(C)]{(C)}}$ we have
$$\sup_{n\ge 1}\E{Q_n^{(a)}\sigma(A)^2}\le C I_{a^2}(\sigma) <+\infty.$$
Hence $Q_n^{(a)}\sigma(A)$ is an $L^2$ martingale and convergence is also in $L^2$ when $a^2<\alpha$. Also it follows that $\E{Q^{(a)}\sigma(D)}=\sigma(D)$. Hence $Q^{(a)}\sigma$ is a non-degenerate measure for $a^2<\alpha$.

{\bf Step 2.} This is the main technical step in which one shows that, for $a^2<\alpha$, $\sigma\in R_\alpha$ implies that $Q^{(a)}\sigma \in R_{\alpha-a^2/2}$. One uses the rooted measure for this scope. Let $M^{(a)}$ be the measure on $D\times \Omega$ defined by
$$M^{(a)}(\De x,\, \De\omega)= Q^{(a)}\sigma(\De x) \mathsf P(\De x)$$
where $\mathsf P$ is the probability on the space where $X$ lives.
We can choose $\sigma$ such that $\sigma(D)=1$ for simplicity and hence we have $M^{(a)}(D\times \Omega)=1$. Let us denote by
$$P_k^{(a)}(x)=\exp\left(aY_k(x)-\frac{a^2}{2}p_k(x,x)\right)$$
and note that $\E{P_k^{(a)}(x)}=1$ for all $k$. It can be shown that $Q_n^{(a)}=\prod_{k=1}^n P_k^{(a)}(x)$ in law.  Also one has the following properties of $Y_k(x)$ under $M^{(a)}$:
\begin{enumerate}
\item[i)] $Y_k(x)\sim\mathcal N(a\,p_k(x,\,x),p_k(x,\,x))$ and they are independent in $k$.
\item[ii)] $\int \log P_k^{(a)}(x)\De M^{(a)}= a^2/2$.
\item[iii)] $\int\left( \log P_k^{(a)}(x)\right)^2\De M^{(a)}= 3a^2/2$.
\end{enumerate}
To see ii) observe that
\begin{align*}
\int \log P_k^{(a)}(x)\De M^{(a)}=\lim_{n\to+\infty} \E{\int_D \log P_k^{(a)}(x) Q_n^{(a)}(x) \sigma(\De x)}=\E{\log P_k^{(a)}(x)P_k^{(a)}(x)}.
\end{align*}
Note that for $h>0$ we have $\E{P_k^{(a)}(x)^h}=\e^{\frac{a^2}{2}(h^2-h)}$. Now deriving with respect to $h$ at $1$ we have $\E{\log P_k^{(a)}(x)P_k^{(a)}(x)}=a^2/2$. iii) follows similarly. Using independence and ii)-iii) we can invoke the strong law of large numbers and say that $M^{(a)}$-almost surely
\begin{equation}
\frac{\log Q_N^{(a)}(x)}{N}=\frac{\sum_{k=1}^N \log P_k^{(a)}(x)}{N}\to \frac{a^2}{2}.
\end{equation}
Hence using Fubini and Egoroff's theorem we get $\mathsf P$-almost surely that for every $\eps>0$, there exist a compact set  $K_\eps^1\subseteq D$ such that $Q^{(a)}\sigma( D\setminus K_\eps^1)<\eps$ and $\log \left(Q_N^{(a)}(x)\right)/N\to a^2/2$  uniformly on $K_\eps^1$.

Let $N\ge 1$ be fixed. Let us define a measure $R_N^{(a)}\sigma(\De x)$ in the following way: let $R_{N,k}^{(a)}\sigma(\De x)=R_{N,k}^{(a)}(x)\sigma(\De x)$ where
$$R_{N,k}^{(a)}(x)=\prod_{i=N+1}^k P_i^{(a)}(x).$$
Then again using a martingale argument it follows that $R_{N,k}^{(a)}\sigma(\De x)$ converges weakly almost surely to a measure, call it $R_N^{(a)}\sigma(\De x)$. Note again that since $a^2<
\alpha$ by condition $\hyperref[(C)]{(C)}$ we obtain $L^2$ convergence as well. Let $B_N(x)$ be a ball of radius $\e^{-N}$ centered at $x$. Define $C_N(x)= R_N^{(a)}\sigma ( B_N(x)\cap D)$. Next we claim \begin{claim}\label{claim2}
For $a^2<\beta<\alpha$, one has $\e^{\beta N}C_N(x)\to 0$  as $N\to+ \infty$ $M^{(a)}$-almost surely.
\end{claim}
By means of Claim~\ref{claim2} we complete the proof of Step 2. From Claim~\ref{claim2} and Egoroff's theorems it follows again that $\mathsf P$-almost surely, there exists a compact set $K_\eps^2\subseteq D$ such that $Q^{a}\sigma( D\setminus K_\eps^2)<\eps$ and $e^{\beta N}C_N(x)\to 0$ as $N\to+\infty$ uniformly on $K_\eps^2$. Let $K_\eps=K_\eps^1\cap K_\eps^2$ and $Q^{(a)}\sigma_K=\one_{K_\eps}(x) Q^{(a)}\sigma(\De x)$. 
$\mathsf P$-almost surely one has for every $N\ge1$
$$Q^{(a)} \sigma(\De x)= Q_N^{(a)}(x) R_N^{(a)}\sigma(\De x).$$
Hence it follows that
\begin{align*}
\limsup_{N\to+\infty}\frac{\log( Q^{(a)}\sigma_K(B_N(x))}{N}&\le \sup_{x\in K_\eps^1}\limsup_{N\to+\infty}\frac{\log Q_N^{(a)}(x)}{N}+\sup_{x\in K_\eps^2}\limsup_{N\to+\infty}\frac{\log C_N(x)}{N}=\frac{a^2}{2}-\beta,
\end{align*}
and the convergence is uniform for all $x\in K_\eps$. Now since $\beta<\alpha$ is arbitrary this entails that $Q^{(a)}\sigma\in R_{\alpha-a^2/2}$. Now the proof of step 2 will be complete if we show Claim~\ref{claim2}. Recall that $C_N(x)=\int_{D}\one_{\|x-y\|\le \e^{-N}}R_N^{(a)}\sigma(\De y)$. First note that using $L^2$ convergence we have
\begin{equation}
\int_{\Omega\times D}C_N^{(a)}(x)\De M^{(a)}=\lim_{k\to+\infty}\int_{D\times D}\one_{\|x-y\|\le \e^{-N}} \E{R_{N,k}^{(a)}(x)R_{N,k}^{(a)}(y)}\sigma(\De y) \sigma(\De x).\label{eq:refer}
\end{equation}
Note that
$$\E{R_{N,k}^{(a)}(x)R_{N,k}^{(a)}(y)}=\e^{a^2(q_k(x,y)-q_N(x,y))}.$$
Thus using \eqref{eq:wn2} there exists a constant $C'$ such that with $k\ge k_0$ we get
\begin{equation}\label{eq:cov_R}
\E{R_{N,k}^{(a)}(x) R_{N,k}^{(a)}(y)}\le C'\e^{-a^2N} \frac1{\|x-y\|^{a^2}} \quad \forall \,k\ge 1.
\end{equation}
So using the above bound one sees that
$$\int_{\Omega\times D}C_N^{(a)}(x)\De M^{(a)}\le C'\int_{D\times D}\e^{-a^2N}\frac1{\|x-y\|^{a^2}}\ \one_{\|x-y\|\le \e^{-N}}\sigma(\De x)\sigma(\De y).$$
Hence for $a^2<\beta<\alpha$ we get, for some constant $C$,
\begin{align*}
&\int\sum_{N=1}^{+\infty} \e^{\beta N} C_N(x) \De M^{(a)}
= C\sum_{N=1}^{+\infty} \e^{(\beta-a^2)N} \int_{D\times D}\frac1{\|x-y\|^{a^2}}\one_{\|x-y\|\le \e^{-N}}\sigma(\De x)\sigma(\De y)\\
&=C\int_{D\times D}\sum_{N=1}^{\lfloor -\log\|x-y\|\rfloor}\e^{(\beta-a^2) N}\frac1{\|x-y\|^{a^2}}\sigma(\De x)\sigma(\De y)\le C(I_{a^2}(\sigma)+ I_\beta(\sigma))<+\infty.
\end{align*}
Here in the last step we have used that there exists a constant $c>0$ such that 
$$\sum_{N=1}^{\lfloor -\log\|x-y\|\rfloor}\e^{(\beta-a^2) N}\le c\left(1- \frac1{\|x-y\|^{\beta-a^2}}\right).$$
Now this proves Claim~\ref{claim2} and hence Step 2.

{\bf Step 3: Iterative limit for $\alpha<a^2<2\alpha$ and final step.} Now we iterate Step 2 to show that if $\lambda$ is the normalized Lebesgue measure on $D$ then for $a^2<2d$, $Q^{(a)}\lambda\in R_{d-a^2/2}$. Assume $a^2>\alpha$. Choose $n$ real numbers $a_1,\,\cdots, \,a_n$ such that 
\eq{*}
\begin{array}{l}
a_1^2<d,\\
a_2^2<d-a_1^2/2,\\
 \cdots\\
 a_n^2<d-\sum_{j=1}^{n-1}a_j^2/2
\end{array}
\eeq{*} 
and $a^2=a_1^2+a_2^2+\cdots+ a_n^2$. Define recursively the following measures: let $S_{0}(\De x)=\lambda(\De x)$ and for $1\le k\le n$ $S_{k}(\De x)$ is the almost sure weak limit of the martingale measure $P_k^{(a)}(x) S_{k-1}(\De x)$. Since $a_k^2<d-\sum_{j=1}^{k-1}{a_j^2}/{2}$ and $S_0\in R_d$ we can apply Step 2 recursively to get that 
$S_k\in R_{d-\sum_{j=1}^{k-1}{a_j^2}/{2}}$ and $\E{S_k(D)}=\E{S_{k-1}(D)}=\lambda(D)$. Note that the later restriction shows that $S_k$ are non-degenerate measures. The finite energy condition now follows by observing that $S_n(\De x)= Q^{(a)}\lambda(\De x)$. Under the measure $M^{(a)}$ we have shown in Step 2 that $Y_k^{(a)}(x)$ are independent and distributed as $\mathcal N(ap_k(x,x),p_k(x,x))$. Since $p_k(x,x)\le c_k$ and $\sum_{k=1}^\infty c_k/k^2<\infty$ by Kolmogorov's strong law of large numbers one has that $M^{(a)}$-almost surely, $\sum_{k=1}^N {Y_k^{(a)}(x)}/{N}= {X_N(x)}/{N}\to a$. Hence $\mathsf P$-almost surely one has
$$\frac{X_N(x)}{N}\to a\quad \text{ $Q^{(a)}\lambda$- almost surely}.$$
This gives that $\mathsf P$-almost surely, $Q^{(a)}\lambda(T(a)^c)=0$. This completes the lower bound.
\end{proof}
\subsection{Proof of Theorem~\ref{thm:comparison}}
Before we start the proof of~Theorem~\ref{thm:comparison}, we state a useful claim which we implement in the proof. 
 \begin{claim}\label{claim:maxima}
Let $\{G_\eps(x), x\in B(0,R), \eps\in (0,1)\}$ be a centered Gaussian process, 
such that for some positive constant $C$
\begin{equation}
\E{(G_\eps(x)-G_\eps(y))^2}\leq C\frac{\norm{x-y}}{\eps}.
\end{equation} 
Then there exists constants $C_1$  (depending only on $C$, $d$ and $R$) such that
$$ \E{\sup_{x\in B(0,R)}G_\eps(x)} \le C_1 \sqrt{-\log \eps}.$$
\end{claim}

\begin{proof}[Proof of Claim~\ref{claim:maxima}]
Without loss of generality let us take $R=1$, $D= B(0,1)$ and let $T(x)$ be a continuous, 
stationary, centered Gaussian process (indexed by $x\in D$) with 
$$\mathrm{Cov}(T(x), T(y))= \frac{\sigma}{\rho}\exp(-\rho\|x-y\|),$$
where $\sigma=2C$ and $\rho$ is some positive constant less than $\eps/2$. Such a Gaussian 
process exists, see for example \citet[Lemma 2.1]{OUSheet}. Using the fact that $1-e^{-x}\ge (x\wedge 1)/2$ we have that
\begin{align*}
&\E{(T(x/\eps)-T( y/\eps))^2}=\frac{\sigma}{\rho}-2\mathrm{Cov}(T(x/\eps),\,T(y/\eps))\\
&=\frac{\sigma}{\rho}\left(1-\exp\left(-\rho\norm{ x-y }\eps^{-1} \right)\right)\ge C\frac{\|x-y\|}{\eps}.
\end{align*}
This shows that $\E{(G_\eps(x)-G_\eps(y))^2}\le \E{(T(x/\eps)-T( y/\eps))^2}$. 
Hence by Sudakov-Fernique's inequality (\citet[Theorem 2.9]{Adl90}), we have that
\begin{equation}\label{eq:sudakov}
\E{\sup_{x\in D}G_\eps(x)}  \leq \E{\sup_{x\in D}T(x/\eps)}=\E{\sup_{x\in B(0,\eps^{-1})}T(x)}.
\end{equation}
Now we can can apply Lemma~11.2 of \citet{chatterjee2008chaos} to conclude that
\[
\E{\sup_{x\in B(0,\eps^{-1})}T(x)}\le C(d)\sqrt{\log N(B(0,1/\eps))},
\]
where, for $A\subseteq \R^d$, $N(A)$ denotes the $1$-packing number. Since it is bounded by the 
$1$-covering number of $B(0,1/\eps)$, it is easy to see that $N(B(0,1/\eps))$ is bounded from 
above by $\eps^{-d}$ and hence the claim now follows from~\eqref{eq:sudakov}.
\end{proof}

Now using Claim~\ref{claim:maxima} we derive a proof of Theorem~\ref{thm:comparison}.
\begin{proof}[Proof of Theorem~\ref{thm:comparison}]
First observe that Assumption~\hyperref[(E)]{(E)} implies we can apply the modified Kolmogorov-Centsov theorem as in Theorem~\ref{thm:upper}  and derive that,
for $x\in D= B(0,R)$ and $\eps\in (0,1]$, there exists a modification $\widetilde Z_\eps(x)$ of
$Z_\eps(x)$ such that for every $\gamma\in (0, 1/2)$ and $a, b>0$ there exists $M>0$ such that
\eq{}\label{eq:modification}\left|\widetilde Z_{\eps_1}(x) - \widetilde Z_{\eps_2}(y)\right| \le 
M \left(\log \frac1{\eps_2}\right)^b\frac{(\Vert x-y\Vert +|\eps_1-\eps_2|)^\gamma}{\eps_2^{(1+a)\gamma}}\eeq{}
for all $x, y\in B(0, R)$ and $\eps_1, \eps_2\in (0,1]$ and $\eps_2/\eps_1\in (1/2,2]$.

We work with a modification of the process and also use the same notation for the process and its modification. First we show that
\begin{equation}\label{eq: limsup}
\limsup_{\eps\to 0} \frac{Z_{\eps}(x)}{-\log \eps}=0.
\end{equation}
From Claim~\ref{claim:maxima} we have that $\E{\sup_{x\in D}Z_\eps(x)}\le C \sqrt{-\log \eps}$. By Borell's inequality (\citet[Chapter 5.1]{Adl90}), 
 \begin{equation}\label{eq:Borell}
\mathbb P\left(\left\lvert\sup_{x\in D}Z_\eps(x)-\E{\sup_{x\in D}Z_\eps(x)}\right\rvert\geq r\right)\leq C\mathrm{e}^{-cr^2/2},
\end{equation}
where $c= (\sup_{x\in D} \E{Z_\eps(x)^2})^{-1}$. Let $a>0$ and  if we choose $\eps_n:=n^{-1/a}$, $r_n=\sqrt{\log{\eps_n^{-3a/c}}}$ then it follows that 
$$\sum_{n=1}^\infty  P\left(\left\lvert\sup_{x\in D}Z_{\eps_n}(x)-\E{\sup_{x\in D}Z_{\eps_n}(x)}\right\rvert\geq r_n\right) \le C\sum_{n=1}^\infty \frac1{n^{3/2}}<+\infty.$$

Now by an easy application of Borel-Cantelli we have that 
$\sup_{x\in D}Z_{\eps_n}(x)= \mathrm{o}(-\log \eps_n)$ almost surely, 
since $\frac{r_n}{-\log \eps_n}\to 0$ as $n\to+\infty$.
Now we claim that due to continuity we can move from the discrete sequence to the continuous sequence. 
We plug in \eqref{eq:modification} the choice of $\eps_n=n^{-1/a}$ and let $\eps\in (\eps_{n+1},\eps_n)$ in order to have
\[
 \left|\sup_{x\in D} Z_{\eps}(x)- \sup_{x\in D} Z_{\eps_n}(x)\right|\le 
\left(\log \frac1{\eps_n}\right)^b\frac{|\eps-\eps_n|^\gamma}{\eps_n^{(1+a)\gamma}}
\le C (\log n)^b =\mathrm{o}(\log n).
\]
This implies that $ |\sup_{x\in D} Z_{\eps}(x)- \sup_{x\in D} Z_{\eps_n}(x)|/-\log \eps_n\to 0$ and hence,
using $Z_\eps(x)= Z_{\eps}(x)- Z_{\eps_n}(x)+ Z_{\eps_n}(x)$ we get that 
$\limsup_{\eps\to 0} {Z_{\eps}(x)}/{-\log \eps}=0$ almost surely.
What is left to show is the equality of the set of thick points, and we begin with the inclusion 
$T(X,a)\subseteq T(\tilde{X},a)$ almost surely. The other follows similarly. Let $x\in T(X,a)$, then 
as a consequence of~\eqref{eq: limsup} it holds that $\limsup_{\eps\to 0} 
\frac{\widetilde X_\eps(x)}{-\log \eps}\ge a$. Since $Z_\eps(x)$ is a symmetric process in $x$, 
we have 
$$\liminf_{\eps\to 0} \frac{\widetilde X_\eps(x)}{-\log \eps}\ge 
\liminf_{\eps\to 0}\left(\inf_{x\in D} \frac{-Z_\eps(x)}{-\log \eps}\right)+\liminf_{\eps\to 0} \frac{X_\eps(x)}{-\log \eps}.$$
Hence using $\liminf_{n}{x_n}=-\limsup_n (-x_n)$ we have 
$$
\liminf_{\eps\to 0} \frac{\widetilde X_\eps(x)}{-\log \eps}\ge\liminf_{\eps\to 0} \frac{X_\eps(x)}{-\log \eps}=a.
$$
This completes the proof of the fact that $T(X, a)\subseteq T(\tilde X, a)$; 
reversing the roles of $X_\eps(x)$ and $\tilde{X}_\eps(x)$ we get the other inclusion to complete the proof.
\end{proof}

\section{Acknowledgments} We are very grateful to two anonymous referees of previous 
versions of the article for suggesting us the question and some ideas on Theorem~\ref{thm:comparison}, and also the example in Subsec.~\ref{subsec:example_comp}. We also thank R\'emi Rhodes and Vincent Vargas for helpful discussions and for providing us with their lecture notes on Gaussian multiplicative chaos.

\bibliographystyle{plainnat}
\bibliography{literatur}

\end{document}